\documentclass{amsart}
\usepackage[a4paper,margin=2.5cm]{geometry}
\usepackage{amssymb,esint}
\usepackage{amscd}
\usepackage{xspace}
\usepackage{fancyhdr}
\usepackage{xcolor}
\usepackage{hyperref,amsmath}
\usepackage{cite}
\newtheorem{theorem}{Theorem}[section]
\newtheorem{lemma}[theorem]{Lemma}

\theoremstyle{definition}

\theoremstyle{proposition}
\newtheorem{proposition}[theorem]{Proposition}
\theoremstyle{remark}
\newtheorem{remark}[theorem]{Remark}

\numberwithin{equation}{section}

\begin{document}

\title{Global Large Solution to Navier-Stokes and SQG Equations with Time Oscillation}

\author{Yiran Xu}
\address{Fudan University, 220 Handan Road, Yangpu, Shanghai, 200433, China.}
\email{yrxu20@fudan.edu.cn}

\author{Haina Li}
\address{School of Mathematics and Statistics, Beijing Institute of Technology, Beijing 100081, China.}
\email{lihaina2000@163.com}





\begin{abstract}
	In this paper, we  consider the model of 3D incompressible Navier-Stokes equations and 2D supercritical Surface Quasi-Geostrophic equations with time oscillation in the nonlinear term. We obtain that there exists global smooth solution of these two equations for any initial data $u_0\in H^2$.
\end{abstract}
\maketitle
\section{Introduction}

We investigate the global well-posedness of the three-dimensional incompressible Navier-Stokes equations with time oscillation in the nonlinear term. The system is described by:
\begin{align}\label{NS}
	\begin{cases}
		\partial_t u-\Delta u+b(Nt)u\cdot \nabla u+\nabla p=0,\\
		\nabla\cdot u=0,\\
		u(t,x)|_{t=0}=u_0(x).
	\end{cases}
\end{align}

Here, \(u\) denotes the velocity field and \(p\) represents the scalar pressure function of the fluid, ensuring the divergence-free condition of the velocity field. The function \(b:[0,+\infty) \rightarrow \mathbb{R}\) satisfies the condition:
\begin{equation}\label{oscillation}
	\|b\|_{L^\infty}+\left|\int_{t_1}^{t_2} b(\tau) d\tau\right|\leq M,\quad \forall t_2>t_1>0,
\end{equation}
where \(M > 0\). It is noteworthy that \(b(t)=\sin t\) satisfies this condition.

Initially, let us revisit the classical Navier-Stokes system for incompressible fluids in three dimension:
\begin{align}\label{NS classical}
	\begin{cases}
		\partial_t u-\Delta u+u\cdot \nabla u+\nabla p=0,\\
		\nabla\cdot u=0,\\
		u(t,x)|_{t=0}=u_0(x),
	\end{cases}
\end{align}
here we simplify the viscosity to 1. 

Notable advancements have been made concerning global existence results. The seminal work by\cite{Leray 1934} demonstrated the existence of at least one global weak solution satisfying energy estimates for any initial data with finite energy. However, the uniqueness of this solution in $\mathbb{R}^3$ remains indeterminate. Fujita-Kato  partially addressed the construction of a globally unique solution in  \cite{Fujita and Kato 1964,Kato 1984}, permitting a unique local-in-time solution with  sufficiently small initial data relative to viscosity in  \(\dot{H}^{\frac{1}{2}}(\mathbb{R}^3)\), \(L^3(\mathbb{R}^3)\) spaces. This result was further extended to $\dot{B}_{p,\infty}^{-1+\frac{3}{p}}$ spaces by Cannone, Meyer and Planchon in \cite{Cannone et al. 1994}. Later, the theorem established by Koch and Tataru in \cite{H. Koch and D. Tataru 2001} (2001) asserts the global well-posedness of the system when \(\partial  BMO\)  norm is small.

Until now, the inquiry into whether smooth initial data with finite energy can lead to a finite-time singularity in the three-dimensional (3D) incompressible Navier-Stokes equations is still an open focal point of the Millennium Prize problems \cite{Charles L. Fefffferman}.

A series of findings leverage the unique structural properties of the Navier-Stokes equations \eqref{NS classical}, where certain geometrical invariances in the initial data persist under the Navier-Stokes flow.  We refer for instance to \cite{15 Ladyzhenskaya}, \cite{Lei FanghuaLin}, \cite{18 Mahalov},
\cite{20 Ponce},  \cite{21 Ukhovskii}, where special symmetries (like helicoidal or axisymmetric without swirl) allow
to prove global well-posedness for any large data.

In addressing the challenges posed by high oscillatory initial data problems in \eqref{NS classical}, Chemin and Ping Zhang established the global existence of solutions \cite{Chemin and Zhang 2007}. Chemin and Gallagher conducted a series of investigations. In \cite{Chemin and Gallagher 2006}, they constructed a global solution with strong oscillation of initial data
\begin{align}\label{os data}
	u_0^N=(Nu_h(x_h)\cos(Nx_3),-div_h u_h(x_h)\sin(Nx_3)),
\end{align}
which was subsequently extended to $\mathbb{R}^3$ in \cite{Chemin and Gallagher 2009}. Later, the gradual variation of initial data in the vertical direction for the Navier-Stokes equations was explored by \cite{Chemin and Gallagher 2010}. More recently, Yanlin Liu and Ping Zhang \cite{Yanlin Liu} demonstrated the existence of a global strong solution with solenoidal initial data for the system \eqref{NS classical}.

In this paper, our central investigation pertains to the global well-posedness of the system delineated by equation \eqref{NS}, wherein time oscillations are present within the nonlinear term. Our methodological approach relies on the bootstrap argument, supplemented by the implementation of a modifier operator.
The principal theorem of our study is expressed as follows:
\begin{theorem}\label{thm1.1}
	For any large data $u_0\in H^2(\mathbb{R}^3)$, there exists $N_0=N_0(||u_0||_{H^2(\mathbb{R}^3)})\geq 1$ such that 
	for any $N\geq N_0$, the equation \eqref{NS} admits global smooth solution. 
\end{theorem}
\begin{remark}
	Our method in proof Theorem \ref{thm1.1} is useful to give a shorter proof of global existence of Navier-Stokes equations \eqref{NS classical} with highly oscillatory data like \eqref{os data}.
\end{remark}

In this manuscript, we also utilize the two-dimensional supercritical Surface Quasi-Geostrophic (SQG) equations with time oscillations to showcase the applicability of our proof methodology. For a comprehensive examination of the global well-posedness of these equations, we direct the reader to the works of \cite{Hongjie Dong 2024},\cite{Hung},\cite{10Constantin},\cite{23  Kiselev},\cite{Nguyen},\cite{27 Resnick}. 

We now precisely delineate the 2D supercritical Surface Quasi-Geostrophic equations under scrutiny in this study:
\begin{align}\label{SQG}
	\begin{cases}
		\partial_t \theta+b(Nt)\nabla^\perp(-\Delta)^{-\frac{1}{2}} \theta\cdot\nabla \theta+(-\Delta)^{\frac{\alpha}{2}} \theta =0,\\
		\theta(t,x)|_{t=0}=\theta_0(x),
	\end{cases}
\end{align}
where $\nabla^\perp=(-\partial_2,\partial_1)$, $\alpha\in(0,1)$ and $\theta=\theta(t,x)$ is a scalar function of $x\in\mathbb{R}^2$ and $t>0$, representing the potential temperature. Here, the function $b$ shares the same conditions as the function utilized in the 3D incompressible Navier-Stokes equations. Leveraging a series of Lemmas, we can readily derive a result analogous to that of the Navier-Stokes equations. The statement of the theorem is presented below:
\begin{theorem}\label{thm 1.2}
	For any large data $\theta_0 \in H^2(\mathbb{R}^2)$, there exists $N_0=N_0(\|\theta_0\|_{H^2(\mathbb{R}^2)}) \geq 1$ such that for any $N\geq N_0$, the equation \eqref{SQG} admits global smooth solution.
\end{theorem}	
\section{Preliminaries}
Now we start to prove some inequalities which will be used several times in the proof of the main propositions. Lemma \ref{interpolation} will be used in Section 3, and Lemma \ref{SQGinterpolation} will be used in Section 4.
\begin{lemma}\label{interpolation}
	Let $u:\mathbb{R}^3 \rightarrow \mathbb{R}^3$ be a vector function, it holds that
	\begin{align}
		&\int_{0}^{T} \|\nabla u\|_{L^\infty(\mathbb{R}^3)} \| u\|_{\dot{H}^2(\mathbb{R}^3)} \|u\|_{\dot{H}^3(\mathbb{R}^3)} dt
		\lesssim \| u\|_{L_T^\infty L^2(\mathbb{R}^3)}^{\frac{1}{4}} \left( \|u\|_{L_T^\infty \dot{H}^2(\mathbb{R}^3)}^{\frac{11}{4}} + \|u\|_{L_T^2 \dot{H}^3(\mathbb{R}^3)}^{\frac{11}{4}}\right),\label{11/4}\\
		&\int_{0}^{T} \|\nabla u\|_{L^\infty(\mathbb{R}^3)} \| u\|_{\dot{H}^1(\mathbb{R}^3)} \|u\|_{\dot{H}^3(\mathbb{R}^3)}   dt \lesssim \|u\|_{L^\infty_T L^2(\mathbb{R}^3)}^{\frac{3}{4}} 
		\left(
		\| u\|_{L^\infty_T\dot{H}^2(\mathbb{R}^3)}^{\frac{9}{4}} + \| u\|_{L_T^2 \dot{H}^3(\mathbb{R}^3)}^{\frac{9}{4}}
		\right),\label{9/4}\\
		&\int_{0}^{T} \|\nabla u\|_{L^\infty(\mathbb{R}^3)}  \| u\|_{\dot{H}^1(\mathbb{R}^3)} \|u \cdot \nabla u\|_{\dot{H}^1(\mathbb{R}^3)}  dt \lesssim \|u\|_{L_T^\infty L^2(\mathbb{R}^3)}^\frac{3}{2} 
		\left(
		\|u\|_{L_T^\infty \dot{H}^2(\mathbb{R}^3)}^{\frac{5}{2}}
		+ \|u\|_{L_T^2 \dot{H}^3(\mathbb{R}^3)}^{\frac{5}{2}}
		\right).\label{5/2}
	\end{align}
\end{lemma}
\begin{proof}
	Before proving the above estimates, we first give the useful inequalities by applying Gagliardo-Nirenberg interpolation.
	\begin{align}
		&\|u\|_{L^\infty(\mathbb{R}^3)} \lesssim \|u\|_{L^2(\mathbb{R}^3)}^{\frac{1}{2}}\|u\|_{\dot{H}^3(\mathbb{R}^3)}^{\frac{1}{2}},\label{u}\\
		&\|\nabla u\|_{L^\infty(\mathbb{R}^3)} \lesssim \|u\|_{L^2(\mathbb{R}^3)}^{\frac{1}{6}} \| u\|_{\dot{H}^3(\mathbb{R}^3)}^{\frac{5}{6}},\label{nabla u}\\
		&\| u\|_{\dot{H}^1(\mathbb{R}^3)} \lesssim \|u\|_{L^2(\mathbb{R}^3)}^{\frac{2}{3}} \| u\|_{\dot{H}^3(\mathbb{R}^3)}^{\frac{1}{3}},\label{nabla u1}\\
		&\| u\|_{\dot{H}^1(\mathbb{R}^3)} \lesssim \|u\|_{L^2(\mathbb{R}^3)}^{\frac{1}{2}}\| u\|_{\dot{H}^2(\mathbb{R}^3)}^{\frac{1}{2}}\label{nabla u2},\\
		&\| u\|_{\dot{H}^2(\mathbb{R}^3)} \lesssim  \|u\|_{L^2(\mathbb{R}^3)}^{\frac{1}{3}} \|u\|_{\dot{H}^3(\mathbb{R}^3)}^{\frac{2}{3}}.\label{nabla^2 u}
	\end{align}
	Combining with \eqref{nabla u}, \eqref{nabla^2 u} and Young's inequality, the following inequality holds true
	\begin{align*}
		\int_{0}^{T} \|\nabla u\|_{L^\infty(\mathbb{R}^3)} \| u\|_{\dot{H}^2(\mathbb{R}^3)} \| u\|_{\dot{H}^3(\mathbb{R}^3)} dt
		&\lesssim \| u\|_{L_T^\infty L^2(\mathbb{R}^3)}^{\frac{1}{4}} \|u\|_{L_T^\infty \dot{H}^2(\mathbb{R}^3)}^{\frac{3}{4}} \| u\|_{L_T^2 \dot{H}^3(\mathbb{R}^3)}^2\\
		&\lesssim \| u\|_{L_T^\infty L^2(\mathbb{R}^3)}^{\frac{1}{4}} \left( \|u\|_{L_T^\infty \dot{H}^2(\mathbb{R}^3)}^{\frac{11}{4}} + \|u\|_{L_T^2 \dot{H}^3(\mathbb{R}^3)}^{\frac{11}{4}}\right),
	\end{align*}
	which completes the proof of \eqref{11/4}.\\
	According to \eqref{nabla u}-\eqref{nabla u2} and Young's inequality, we obtain
	\begin{align*}
		\int_{0}^{T}  \|\nabla u\|_{L^\infty(\mathbb{R}^3)}  \| u\|_{\dot{H}^1(\mathbb{R}^3)} \| u\|_{\dot{H}^3(\mathbb{R}^3)} dt
		&\lesssim \|u\|_{L^\infty_T L^2(\mathbb{R}^3)}^{\frac{3}{4}} \|u\|_{L^\infty_T \dot{H}^2(\mathbb{R}^3)}^{\frac{1}{4}} \| u\|_{L_T^2 \dot{H}^3(\mathbb{R}^3)}^2\\
		&\lesssim \|u\|_{L^\infty_T L^2(\mathbb{R}^3)}^{\frac{3}{4}}
		\left(
		\| u\|_{L^\infty_T\dot{H}^2(\mathbb{R}^3)}^{\frac{9}{4}} + \| u\|_{L_T^2 \dot{H}^3(\mathbb{R}^3)}^{\frac{9}{4}}
		\right),
	\end{align*}
	hence we obtain the estimate in \eqref{9/4}.\\
	On account of \eqref{u}-\eqref{nabla u1} and \eqref{nabla^2 u}, we have the following results with the application of paraproduct theory\cite{99Bahouri} and Young's inequality,
	\begin{align*}
		&\int_{0}^{T} \|\nabla u\|_{L^\infty(\mathbb{R}^3)} \| u\|_{\dot{H}^1(\mathbb{R}^3)}  \|u \cdot \nabla u\|_{\dot{H}^1(\mathbb{R}^3)}  dt\\
		&\quad\leq \int_{0}^{T}  \|\nabla u\|_{L^\infty(\mathbb{R}^3)} \| u\|_{\dot{H}^1(\mathbb{R}^3)}
		\left(\|u\|_{L^\infty(\mathbb{R}^3)} \|u\|_{\dot{H}^2(\mathbb{R}^3)} + \|u\|_{\dot{H}^1} \|\nabla u\|_{L^\infty(\mathbb{R}^3)}\right) dt\\
		&\quad\lesssim \int_{0}^{T} \left(\|u\|_{L^2(\mathbb{R}^3)}^{\frac{3}{2}} \|u\|_{\dot{H}^2(\mathbb{R}^3)}^\frac{1}{2} \|u\|_{\dot{H}^3(\mathbb{R}^3)}^2  
		+ \|u\|_{\dot{H}^1(\mathbb{R}^3)}^2 \|\nabla u\|_{L^\infty(\mathbb{R}^3)}^2\right) dt\\
		&\quad\lesssim \|u\|_{L_T^\infty L^2(\mathbb{R}^3)}^\frac{3}{2} \|u\|_{L_T^\infty \dot{H}^2(\mathbb{R}^3)}^{\frac{1}{2}} \|u\|_{L_T^2 \dot{H}^3(\mathbb{R}^3)}^2\\
		&\quad\lesssim \|u\|_{L_T^\infty L^2(\mathbb{R}^3)}^\frac{3}{2} 
		\left(
		\|u\|_{L_T^\infty \dot{H}^2(\mathbb{R}^3)}^{\frac{5}{2}}
		+ \|u\|_{L_T^2 \dot{H}^3(\mathbb{R}^3)}^{\frac{5}{2}}
		\right),
	\end{align*}
	therefore we complete the proof of \eqref{5/2}.
\end{proof}
\begin{lemma}\label{SQGinterpolation}
	Let $\theta:\mathbb{R}^2 \rightarrow \mathbb{R}$  be scalar function, $\alpha \in (0,1)$, it holds that
	\small{\begin{align}
			&\int_{0}^{T} \|\theta\|_{\dot{H}^{2+\frac{\alpha}{2}}(\mathbb{R}^2)}
			\|\nabla \theta\|_{L^\infty(\mathbb{R}^2)} \|\theta\|_{\dot{H}^2(\mathbb{R}^2)} dt
			\lesssim
			\|\theta\|_{L_T^\infty \dot{H}^{\frac{\alpha}{2}} (\mathbb{R}^2)}^{\frac{\alpha}{4}}
			\left(
			\|\theta\|_{L_T^2 \dot{H}^{2+\frac{\alpha}{2}}(\mathbb{R}^2)}^{\frac{12-\alpha}{4}} + \|\theta\|_{L_T^2 \dot{H}^2(\mathbb{R}^2)}^{\frac{12-\alpha}{4}}
			\right),\label{estimate 1}\\
			&\int_{0}^T 
			\|\theta\|_{\dot{H}^{1+\alpha}(\mathbb{R}^2)} 
			\|\nabla \theta\|_{L^\infty (\mathbb{R}^2)}
			\|\theta\|_{\dot{H}^1(\mathbb{R}^2)} dt
			\lesssim \|\theta\|_{L_T^\infty \dot{H}^{\frac{\alpha}{2}}(\mathbb{R}^2)}^{\frac{1}{2}}
			\|\theta\|_{L_T^\infty L^2(\mathbb{R}^2)}^{\frac{1}{2}}
			\left(
			\|\theta\|_{L_T^2 \dot{H}^{2+\frac{\alpha}{2}}(\mathbb{R}^2)}^2
			+ \|\theta\|_{L_T^2 \dot{H}^2(\mathbb{R}^2)}^2
			\right),\label{estimate 2}\\
			&\int_{0}^T
			\|\nabla \theta\|_{L^\infty(\mathbb{R}^2)}^2
			\| \theta\|_{\dot{H}^1(\mathbb{R}^2)}^2 dt
			\lesssim
			\|\theta\|_{L_T^\infty L^2(\mathbb{R}^2)}
			\|\theta\|_{L_T^\infty \dot{H}^{\frac{\alpha}{2}}(\mathbb{R}^2)}^{\frac{\alpha}{2}}
			\left(
			\|\theta\|_{L_T^2\dot{H}^{2+\frac{\alpha}{2}}(\mathbb{R}^2)}^{\frac{6-\alpha}{2}}
			+ \|\theta\|_{L_T^2 \dot{H}^2(\mathbb{R}^2)}^{\frac{6-\alpha}{2}}
			\right),\label{estimate 3}\\
			& \int_{0}^T
			\|\nabla \theta\|_{L^\infty(\mathbb{R}^2)}^2
			\|\theta\|_{L^2 (\mathbb{R}^2)}
			\|\theta\|_{\dot{H}^2(\mathbb{R}^2)} dt
			\lesssim \|\theta\|_{L_T^\infty L^2(\mathbb{R}^2)}
			\|\theta\|_{L_T^\infty \dot{H}^{\frac{\alpha}{2}}(\mathbb{R}^2)}^{\frac{\alpha}{2}}
			\left(
			\|\theta\|_{L_T^2 \dot{H}^{2+\frac{\alpha}{2}}(\mathbb{R}^2)}^{\frac{6-\alpha}{2}}
			+ \|\theta\|_{L_T^2 \dot{H}^2(\mathbb{R}^2)}^{\frac{6-\alpha}{2}}
			\right).\label{estimate 4}
	\end{align}	}
\end{lemma}		
\begin{proof}
	To prove the above estimates, we need a series of inequalities equipped with Gagliardo-Nirenberg interpolation.
	\begin{align}
		&\|\nabla \theta\|_{L^\infty(\mathbb{R}^2)}
		\lesssim \|\theta\|_{\dot{H}^{2+\frac{\alpha}{2}}(\mathbb{R}^2)}
		^{1-\frac{\alpha}{4}}
		\|\theta\|_{\dot{H}^{\frac{\alpha}{2}}(\mathbb{R}^2)}^{\frac{\alpha}{4}},\label{nabla theta}\\
		&\|\theta\|_{\dot{H}^1(\mathbb{R}^2)}
		\lesssim \|\theta\|_{\dot{H}^2(\mathbb{R}^2)}^{\frac{1}{2}}
		\|\theta\|_{L^2(\mathbb{R}^2)}^{\frac{1}{2}},\label{nabla theta1}\\
		&\|\theta\|_{\dot{H}^{1+\alpha}(\mathbb{R}^2)}
		\lesssim \|\theta\|_{\dot{H}^{2+\frac{\alpha}{2}}(\mathbb{R}^2)}^{\frac{1}{2}+\frac{\alpha}{4}}
		\|\theta\|_{\dot{H}^{\frac{\alpha}{2}}(\mathbb{R}^2)}^{\frac{1}{2}-\frac{\alpha}{4}}.\label{nabla theta2}
	\end{align}	
	Combining with \eqref{nabla theta} and Young's inequality, the following inequality holds true
	\begin{align*}
		\int_{0}^{T} \|\theta\|_{\dot{H}^{2+\frac{\alpha}{2}}(\mathbb{R}^2)}
		\|\nabla \theta\|_{L^\infty(\mathbb{R}^2)} \|\theta\|_{\dot{H}^2(\mathbb{R}^2)} dt
		&\leq \|\theta\|_{L_T^\infty \dot{H}^{\frac{\alpha}{2}}(\mathbb{R}^2)} ^{\frac{\alpha}{4}}
		\|\theta\|_{L_T^2 \dot{H}^{2+\frac{\alpha}{2}}(\mathbb{R}^2)}^{2-\frac{\alpha}{4}}
		\|\theta\|_{L_T^2 \dot{H}^2(\mathbb{R}^2)}\\
		&\lesssim   \|\theta\|_{L_T^\infty \dot{H}^{\frac{\alpha}{2}} (\mathbb{R}^2)}^{\frac{\alpha}{4}}
		\left(
		\|\theta\|_{L_T^2 \dot{H}^{2+\frac{\alpha}{2}}(\mathbb{R}^2)}^{\frac{12-\alpha}{4}} + \|\theta\|_{L_T^2 \dot{H}^2(\mathbb{R}^2)}^{\frac{12-\alpha}{4}}
		\right),
	\end{align*}	
	which completes the proof of \eqref{estimate 1}.\\
	According to \eqref{nabla theta}-\eqref{nabla theta2} and Young's inequality, we obtain
	\begin{align*}
		\int_{0}^T 
		\|\theta\|_{\dot{H}^{1+\alpha}(\mathbb{R}^2)} 
		\|\nabla \theta\|_{L^\infty (\mathbb{R}^2)}
		\|\theta\|_{\dot{H}^1(\mathbb{R}^2)} dt
		&\leq \|\theta\|_{L_T^\infty \dot{H}^{\frac{\alpha}{2}}(\mathbb{R}^2)}^{\frac{1}{2}}
		\|\theta\|_{L_T^\infty L^2(\mathbb{R}^2)}^{\frac{1}{2}}
		\|\theta\|_{L_T^2 \dot{H}^{2+\frac{\alpha}{2}}(\mathbb{R}^2)}^{\frac{3}{2}}
		\|\theta\|_{L_T^2 \dot{H}^2 (\mathbb{R}^2)}^{\frac{1}{2}}\\
		&\lesssim \|\theta\|_{L_T^\infty \dot{H}^{\frac{\alpha}{2}}(\mathbb{R}^2)}^{\frac{1}{2}}
		\|\theta\|_{L_T^\infty L^2(\mathbb{R}^2)}^{\frac{1}{2}}
		\left(
		\|\theta\|_{L_T^2 \dot{H}^{2+\frac{\alpha}{2}}(\mathbb{R}^2)}^2
		+ \|\theta\|_{L_T^2 \dot{H}^2(\mathbb{R}^2)}^2
		\right),
	\end{align*}	
	hence we obtain the estimate in \eqref{estimate 2}.\\
	On account of \eqref{nabla theta} and \eqref{nabla theta1}, we have the following results with the application of Young's inequality, 
	\begin{align*}
		\int_{0}^T
		\|\nabla \theta\|_{L^\infty(\mathbb{R}^2)}^2
		\|\theta\|_{\dot{H}^1(\mathbb{R}^2)}^2 dt
		&\leq \|\theta\|_{L_T^\infty \dot{H}^{\frac{\alpha}{2}}(\mathbb{R}^2)}^{\frac{\alpha}{2}}
		\|\theta\|_{L_T^\infty L^2(\mathbb{R}^2)}
		\|\theta\|_{L_T^2 \dot{H}^{2+\frac{\alpha}{2}}(\mathbb{R}^2)}^{2-\frac{\alpha}{2}}
		\|\theta\|_{L_T^2 \dot{H}^2(\mathbb{R}^2)}\\
		&\lesssim \|\theta\|_{L_T^\infty \dot{H}^{\frac{\alpha}{2}}(\mathbb{R}^2)}^{\frac{\alpha}{2}}
		\|\theta\|_{L_T^\infty L^2(\mathbb{R}^2)}
		\left(
		\|\theta\|_{L_T^2\dot{H}^{2+\frac{\alpha}{2}}(\mathbb{R}^2)}^{\frac{6-\alpha}{2}}
		+ \|\theta\|_{L_T^2 \dot{H}^2(\mathbb{R}^2)}^{\frac{6-\alpha}{2}}
		\right),
	\end{align*}	
	therefore we complete the proof of \eqref{estimate 3}.\\
	Due to \eqref{nabla theta}, \eqref{nabla theta1} and Young's inequality, we arrive at
	\begin{align*}
		\int_{0}^T
		\|\nabla \theta\|_{L^\infty(\mathbb{R}^2)}^2
		\|\theta\|_{L^2 (\mathbb{R}^2)}
		\|\theta\|_{\dot{H}^2(\mathbb{R}^2)} dt
		&\leq \|\theta\|_{L_T^\infty L^2 (\mathbb{R}^2)}
		\|\theta\|_{L_T^\infty \dot{H}^{\frac{\alpha}{2}}(\mathbb{R}^2)}^{\frac{\alpha}{2}}
		\|\theta\|_{L_T^2\dot{H}^{2+\frac{\alpha}{2}}(\mathbb{R}^2)}^{2-\frac{\alpha}{2}}
		\|\theta\|_{L_T^2\dot{H}^2(\mathbb{R}^2)} \\
		&\lesssim \|\theta\|_{L_T^\infty L^2(\mathbb{R}^2)}
		\|\theta\|_{L_T^\infty
			\dot{H}^{\frac{\alpha}{2}}(\mathbb{R}^2)}^{\frac{\alpha}{2}}
		\left(
		\|\theta\|_{L_T^2 \dot{H}^{2+\frac{\alpha}{2}}(\mathbb{R}^2)}^{\frac{6-\alpha}{2}}
		+ \|\theta\|_{L_T^2 \dot{H}^2(\mathbb{R}^2)}^{\frac{6-\alpha}{2}}
		\right),
	\end{align*}	
	thus we have done the proof of \eqref{estimate 4}.
\end{proof}	
\begin{lemma}\label{epsilon}
	Let $f_\epsilon(x)=\left(f\star \rho_\epsilon\right)(x)$, where $\rho_\epsilon(x)=\epsilon^{-d}\rho(x/\epsilon)$ is a modifier, then for $m,m_1,m_2 \in \mathbb{N}$ satisfying $m=m_1+m_2$, we have the following estimates:
	\begin{align*}
		&\|f_\epsilon\|_{\dot{H}^m(\mathbb{R}^d)}\leq \epsilon^{-m_2} \|f\|_{\dot{H}^{m_1}(\mathbb{R}^d)},\\
		&\|\nabla^m f_\epsilon\|_{L^\infty(\mathbb{R}^d)} \leq \epsilon^{-m_2} \|\nabla^{m_1} f\|_{L^\infty(\mathbb{R}^d)}.
	\end{align*}
\end{lemma}
\begin{proof}
	It follows from the property of modifier that
	\begin{align*}
		\| f_\epsilon\|_{\dot{H}^m(\mathbb{R}^d)}
		&\leq \left(
		\int_{\mathbb{R}_x^d} \left(\int_{\mathbb{R}_y^d} |\nabla_x^{m_1} f(x-y)|^2 |\nabla_y^{m_2} \rho_\epsilon(y)|dy\right) \cdot \left(\int_{\mathbb{R}_y^d} |\nabla_y^{m_2} \rho_\epsilon(y)|dy\right) dx
		\right)^{\frac{1}{2}}\\
		&\leq \epsilon^{-\frac{m_2}{2}}\| f\|_{\dot{H}^{m_1}(\mathbb{R}^d)} \left(\int_{\mathbb{R}_y^d}|\nabla_y^{m_2} \rho_\epsilon(y)|dy\right)^{\frac{1}{2}}\\
		&\leq \epsilon^{-m_2} \| f\|_{\dot{H}^{m_1}(\mathbb{R}^d)},
	\end{align*}
	where in the first line we have used H\"{o}lder inequality.\\
	Similarly, it is clear that
	\begin{align*}
		\|\nabla^m f^\epsilon\|_{L^\infty(\mathbb{R}^d)}
		\leq\|\nabla^{m_1} f\|_{L^\infty(\mathbb{R}^d)} \left\|\int_{\mathbb{R}_y^d} \nabla^{m_2}_y \rho_\epsilon(y) dy\right\|_{L^\infty}
		\leq \epsilon^{-m_2} \|\nabla^{m_1} f\|_{L^\infty(\mathbb{R}^d)},
	\end{align*}
	which completes the proof.
\end{proof} 
\begin{remark}
	In \cite{Tang}, authors also yield similar results  in Lemma A.1.
\end{remark}	
Next, we give the following inequality which will be used in the proof of Proposition \ref{proposition1} and  Proposition  \ref{proposition2}.
\begin{lemma}\label{modifier}
	Let $f_\epsilon(x)$ be mentioned in Lemma \ref{epsilon}, then it holds that for $\forall s\in(0,1]$,
	\begin{align*}
		\|f-f_\epsilon\|_{L^2(\mathbb{R}^d)}\leq \epsilon^s \|f\|_{\dot{H}^s(\mathbb{R}^d)}.
	\end{align*}
\end{lemma}
\begin{proof}
	The trick of the proof is the application of Minkowski's inequality. Due to the definition of modifier and fractional Sobolev space, it is evident that for $s\in(0,1)$,
	\begin{align*}
		\|f-f_\epsilon\|_{L^2(\mathbb{R}^d)}
		&=\left(
		\int_{\mathbb{R}_x^d} \left|\int_{\mathbb{R}_y^d} \left(f(x)-f(x-y)\right) \rho_\epsilon(y) dy\right|^2 dx
		\right)^{\frac{1}{2}}\\
		&\leq \int_{\mathbb{R}_y^d}  |y|^{\frac{d}{2}+s} \rho_\epsilon(y)
		\left(
		\int_{\mathbb{R}_x^d} \frac{\left|f(x)-f(x-y)\right|^2 }{|y|^{d+2s}}dx
		\right)^{\frac{1}{2}}  dy\\
		&\leq\left(\int_{\mathbb{R}_y^d} \int_{\mathbb{R}_x^d}
		\frac{\left|f(x)-f(x-y)\right|^2 }{|y|^{d+2s}} dx dy
		\right)^{\frac{1}{2}} 
		\left(
		\int_{\mathbb{R}_y^d} |y|^{d+2s} \rho_\epsilon^2(y) dy
		\right)^{\frac{1}{2}} \\
		&=\epsilon^s \|f\|_{\dot{H}^s(\mathbb{R}^d)}.
	\end{align*}
	Moreover, for $s=1$, we obtain directly by Newton-Leibniz formula,
	\begin{align*}
		\|f-f_\epsilon\|_{L^2(\mathbb{R}^d)} 
		&\leq \epsilon\int_{\mathbb{R}_y^d} \rho_\epsilon(y)\left(
		\int_{\mathbb{R}_x^d} \left(\int_{0}^{1} \nabla f(x+ty)dt\right)^2 dx 
		\right)^{\frac{1}{2}}  dy\\
		&\leq \epsilon \int_{\mathbb{R}_y^d} \rho_\epsilon(y)\left(\int_{0}^{1} \left(\int_{\mathbb{R}_x^d} \left(\nabla f(x+ ty)\right)^2 dx\right)^{\frac{1}{2}} dt\right)  dy\\
		&=\epsilon \|\nabla f\|_{L^2(\mathbb{R}^d)}.
	\end{align*}
	Hence, an easy result is given.
\end{proof}
\section{Global large solution to  3D incompressible Navier-Stokes equations with time oscillation}
\qquad
Before giving the proof of Theorem \ref{thm1.1}, we firstly give a useful proposition which is indispensable in the theorem proving process.
 For simplicity, we introduce  the following notation:
\begin{align*}
	L_T^pX=L^p(0,T;X).
\end{align*}
\begin{proposition} \label{proposition1}
	Let $u$ be a solution of \eqref{NS} and denote
	\begin{equation}
		X_T=\|u\|_{L^\infty_T\dot{H}^2(\mathbb{R}^3)}^2+\|u\|_{L^2_T\dot{H}^3(\mathbb{R}^3)}^2,
	\end{equation}
	then we have 
	\begin{equation}
		X_T\lesssim \|u_0\|_{\dot{H}^2(\mathbb{R}^3)}^2+\frac{1}{\sqrt[3]{N}}\left(\|u_0\|_{ L^2(\mathbb{R}^3)}^{\frac{3}{4}} 
		X_T^{\frac{9}{8}}+\|u_0\|_{ L^2(\mathbb{R}^3)}^\frac{3}{2} 
		X_T^{\frac{5}{4}}+\|u_0\|_{H^2(\mathbb{R}^3)}^3\right)^{\frac{1}{3}}\|u_0\|_{L^2(\mathbb{R}^3)}^{\frac{1}{6}} X_T^{\frac{11}{12}}.
	\end{equation}
\end{proposition}
\begin{proof}
	Multiply $\Delta^2 u$ as test function on \eqref{NS} and integral over $dx$, we get 
	\begin{align}\label{dx1}
		\frac{1}{2}\partial_t \|\Delta u(t)\|_{L^2(\mathbb{R}^3)}^2
		+\|\Delta u(t)\|_{\dot{H}^1(\mathbb{R}^3)}^2
		=-b(Nt)\int_{\mathbb{R}^3}\nabla \cdot (\Delta(u\otimes u))\cdot \Delta u dx- \int_{\mathbb{R}^3}\nabla p\cdot \Delta^2 u dx.
	\end{align}
	Since $u$ is divergence free, we obtain that the second term on right hand-side equals $0$, i.e.
	\begin{align*}
		\int_{\mathbb{R}^3}\nabla p\cdot \Delta^2 u dx=-\int_{\mathbb{R}^3} p\cdot \Delta^2 (\nabla\cdot u) dx=0.	
	\end{align*}
	For the first term on the right hand-side, we compute as follows:
	\begin{align*}
		\int_{\mathbb{R}^3}\nabla \cdot (\Delta(u\otimes u))\cdot \Delta u dx&=\sum_{i,j}\int_{\mathbb{R}^3} \partial_i \left(\Delta(u^j u^i)\right)\Delta u^j dx\\
		&=\sum_{i,j} \int_{\mathbb{R}^3} \left(\partial_i \left(\Delta u^j u^i\right)\Delta u^j+\partial_i \left( u^j \Delta u^i\right)\Delta u^j+\partial_i \left(\nabla  u^j \cdot\nabla u^i\right)\Delta u^j\right) dx\\
		&=\sum_{i,j} \int_{\mathbb{R}^3} \left(\partial_i \left(\Delta u^j \right)u^i\Delta u^j+\partial_i u^i \Delta u^j\Delta u^j+\partial_i \left(  \Delta u^i\right)u^j\Delta u^j\right.\\
		&\qquad	\left. +\partial_i u^j\Delta u^i\Delta u^j
		+ \nabla (\partial_i u^j)\cdot \nabla u^i\Delta u^j+\nabla (\partial_i u^i)\cdot \nabla u^j\Delta u^j\right) dx.
	\end{align*}
	Applying divergence free property again, one has
	\begin{align*}
		\sum_{i,j}\partial_i \left(  \Delta u^i\right)u^j\Delta u^j=\sum_{j}\Delta\left(\sum_{i} \partial_iu^i\right)u^j\Delta u^j=0.
	\end{align*}
	We also note that
	\begin{align*}
		\int_{\mathbb{R}^3}\partial_i \left(\Delta u^j \right)u^i\Delta u^j dx=-	\int_{\mathbb{R}^3} \Delta u^j \partial_iu^i\Delta u^j dx -	\int_{\mathbb{R}^3}\Delta u^j u^i\partial_i \left(\Delta u^j \right)dx,
	\end{align*}
	which implies that
	\begin{align*}
		\int_{\mathbb{R}^3}\partial_i \left(\Delta u^j \right)u^i\Delta u^j dx=-\frac{1}{2}	\int_{\mathbb{R}^3} \Delta u^j \partial_iu^i\Delta u^j dx.
	\end{align*}
	Thus we have
	\begin{align*}
		&	\int_{\mathbb{R}^3}\nabla \cdot (\Delta(u\otimes u))\cdot \Delta u dx\\
		&\qquad=\sum_{i,j} \int_{\mathbb{R}^3} \left(\frac{1}{2}\partial_i u^i \Delta u^j\Delta u^j+\partial_i u^j\Delta u^i\Delta u^j
		+ \nabla (\partial_i u^j)\cdot \nabla u^i\Delta u^j+\nabla (\partial_i u^i)\cdot \nabla u^j\Delta u^j\right) dx.
	\end{align*}
	
	Now we take integral over time about the equality \eqref{dx1}, we have
	\begin{align*}
		&	\frac{1}{2}\|\Delta u(T)\|_{L^2(\mathbb{R}^3)}^2	-\frac{1}{2}\|\Delta u_0\|_{L^2(\mathbb{R}^3)}^2+\int_{0}^{T}\|\Delta u(t)\|_{\dot{H}^1(\mathbb{R}^3)}^2\\
		&=-\sum_{i,j}\int_{0}^{T}\int_{\mathbb{R}^3} b(Nt)\left(\frac{1}{2}\partial_i u^i \Delta u^j\Delta u^j+\partial_i u^j\Delta u^i\Delta u^j
		+ \nabla (\partial_i u^j)\cdot \nabla u^i\Delta u^j+\nabla (\partial_i u^i)\cdot \nabla u^j\Delta u^j\right)  dxdt
	\end{align*}
	Let $u_\epsilon(t,x)=\left(u(t)\star \rho_\epsilon\right)(x)$, where $\rho_\epsilon(x)=\epsilon^{-3}\rho(x/\epsilon)$ is a modifier. Then we divide the right hand-side into two parts,
	\begin{align*}
		RHS=\mathcal{I}+\mathcal{J},
	\end{align*}
	where
	\begin{align*}
		\mathcal{I}:=&-\sum_{i,j}\int_{0}^{T}\int_{\mathbb{R}^3} b(Nt)\bigg(\frac{1}{2}\partial_i u^i \left(\Delta u^j\Delta u^j-\Delta u^j_\epsilon\Delta u^j_\epsilon\right)+\partial_i u^j\left(\Delta u^i\Delta u^j-\Delta u^i_\epsilon\Delta u^j_\epsilon\right)\bigg.\\
		&\qquad\bigg.+ \left(\Delta u^j\nabla (\partial_i u^j)-\Delta u^j_\epsilon\nabla (\partial_i u^j_\epsilon)\right)\cdot \nabla u^i+\left(\Delta u^j\nabla (\partial_i u^i)-\Delta u^j_\epsilon\nabla (\partial_i u^i_\epsilon)\right)\cdot \nabla u^j\bigg)dxdt,\\
		\mathcal{J}:=&-\sum_{i,j}\int_{0}^{T}\int_{\mathbb{R}^3} b(Nt)\left(\frac{1}{2}\partial_i u^i \Delta u^j_\epsilon\Delta u^j_\epsilon+\partial_i u^j\Delta u^i_\epsilon\Delta u^j_\epsilon
		+ \nabla (\partial_i u^j_\epsilon)\cdot \nabla u^i\Delta u^j_\epsilon+\nabla (\partial_i u^i_\epsilon)\cdot \nabla u^j\Delta u^j_\epsilon\right)  dxdt.
	\end{align*}
	Applying the fact that $	|f-f_\epsilon|\lesssim\epsilon\sup|\nabla f|$, we firstly estimate the first term  of $\mathcal{I}$ as follows,
	\begin{align*}
		&-\sum_{i,j}\int_{0}^{T}\int_{\mathbb{R}^3}\frac{1}{2} b(Nt)\partial_i u^i \left(\Delta u^j\Delta u^j-\Delta u^j_\epsilon\Delta u^j_\epsilon\right)dxdt\lesssim M\sum_{i,j}\int_{0}^{T}\int_{\mathbb{R}^3} |\partial_i u^i| \left|\Delta u^j\Delta u^j-\Delta u^j_\epsilon\Delta u^j_\epsilon\right|dxdt\\
		&\qquad\lesssim M\sum_{i,j}\int_{0}^{T}\int_{\mathbb{R}^3} |\partial_i u^i|\left( \left|\Delta u^j(\Delta u^j-\Delta u^j_\epsilon)\right|+\left|\Delta u^j_\epsilon(\Delta u^j-\Delta u^j_\epsilon)\right|\right)dxdt.
	\end{align*}
	For the first sum on the right hand-side, applying H\"{o}lder inequality and Lemma \ref{modifier} we can deduce that
	\begin{align*}
		\sum_{i,j}\int_{0}^{T}\int_{\mathbb{R}^3} |\partial_i u^i|\left|\Delta u^j(\Delta u^j-\Delta u^j_\epsilon)\right|dxdt
		&\leq \sum_{i,j}\int_{0}^{T}\|\partial_i u^i\|_{L^\infty(\mathbb{R}^3)}\|\Delta u^j\|_{L^2(\mathbb{R}^3)}\|\Delta u^j-\Delta u^j_\epsilon\|_{L^2(\mathbb{R}^3)} dt\\
		&\leq \epsilon \sum_{i,j}\int_{0}^{T}\|\partial_i u^i\|_{L^\infty(\mathbb{R}^3)}\|\Delta u^j\|_{L^2(\mathbb{R}^3)}\|\nabla^3 u^j\|_{L^2(\mathbb{R}^3)} dt.
	\end{align*}
	Concerning the second sum on the right hand-side, an argument similar to the first sum shows that
	\begin{align*}
		\sum_{i,j}\int_{0}^{T}\int_{\mathbb{R}^3} |\partial_i u^i|\left|\Delta u^j_\epsilon(\Delta u^j-\Delta u^j_\epsilon)\right|dxdt
		&\leq \epsilon \sum_{i,j}\int_{0}^{T}\|\partial_i u^i\|_{L^\infty(\mathbb{R}^3)}\|\Delta u_\epsilon^j\|_{L^2(\mathbb{R}^3)}\|\nabla^3 u^j\|_{L^2(\mathbb{R}^3)} dt.
	\end{align*}
	Using similar estimate, we obtain that 
	\begin{align*}
		\mathcal{I}
		&\lesssim \epsilon M \int_{0}^{T} \|\nabla u\|_{L^\infty(\mathbb{R}^3)} \|\nabla^2 u\|_{L^2(\mathbb{R}^3)} \|\nabla^3 u\|_{L^2(\mathbb{R}^3)} dt\\
		&\lesssim \epsilon M \| u\|_{L_T^\infty L^2(\mathbb{R}^3)}^{\frac{1}{4}} \left( \|\nabla^2u\|_{L_T^\infty L^2(\mathbb{R}^3)}^{\frac{11}{4}} + \|\nabla^3u\|_{L_T^2 L^2(\mathbb{R}^3)}^{\frac{11}{4}}\right)\\
		&\leq \epsilon M \|u_0\|_{L^2(\mathbb{R}^3)}^{\frac{1}{4}} X_T^{\frac{11}{8}},
	\end{align*} 
	where in the second inequality we use \eqref{11/4}.\\
	For the term $\mathcal{J}$, without loss of generality, we only show the estimate of the third term of $\mathcal{J}$. We integral by parts over $dt$ to get
	\begin{align*}
		&	-\int_{0}^{T}\int_{\mathbb{R}^3} b(Nt) \nabla (\partial_i u^j_\epsilon)\cdot \nabla u^i\Delta u^j_\epsilon dxdt
		=-\frac{1}{N} 	\int_{0}^{T}\int_{\mathbb{R}^3} \partial_t\left(\int_{0}^{t} b(N\tau)d\tau\right) \nabla (\partial_i u^j_\epsilon)\cdot \nabla u^i\Delta u^j_\epsilon dxdt\\
		&=\frac{1}{N} \int_{0}^{T}\int_{\mathbb{R}^3} \left(\int_{0}^{t} b(N\tau)d\tau\right)\partial_t\left[ \nabla (\partial_i u^j_\epsilon)\cdot \nabla u^i\Delta u^j_\epsilon \right]dxdt
		-\frac{1}{N} 	\int_{\mathbb{R}^3} \left(\int_{0}^{t} b(N\tau)d\tau\right) \nabla (\partial_i u^j_\epsilon)\cdot \nabla u^i\Delta u^j_\epsilon dx\vert_{t=0}^{t=T}\\
		&=:\mathcal{J}^1+\mathcal{J}^2.
	\end{align*}
	As for the term $\mathcal{J}^1$, we write
	\small{\begin{align*}
			\mathcal{J}^1
			=&\frac{1}{N} 	\int_{0}^{T}\int_{\mathbb{R}^3} \left(\int_{0}^{t} b(N\tau)d\tau\right)\nabla (\partial_i u^j_\epsilon)\cdot \nabla (\partial_tu^i)\Delta u^j_\epsilon dxdt
			+\frac{1}{N} 	\int_{0}^{T}\int_{\mathbb{R}^3} \left(\int_{0}^{t} b(N\tau)d\tau\right)\nabla (\partial_i \partial_tu^j_\epsilon)\cdot \nabla u^i\Delta u^j_\epsilon dxdt\\
			&+\frac{1}{N} 	\int_{0}^{T}\int_{\mathbb{R}^3} \left(\int_{0}^{t} b(N\tau)d\tau\right)\nabla (\partial_i u^j_\epsilon)\cdot \nabla (u^i)\Delta(\partial_t u^j_\epsilon) dxdt=:	\mathcal{J}^1_1+	\mathcal{J}^1_2+	\mathcal{J}^1_3.
	\end{align*}}
	To treat the term $\mathcal{J}_1^1$,  we make use of the fact that
	$\partial_t u=\Delta u-\mathbb{P}(b(Nt)u\cdot \nabla u)$,
	where $\mathbb{P}f:=f-\nabla \Delta^{-1}(\nabla \cdot f)$ representing Leray projection. Thus we have
	\begin{align*}
		\mathcal{J}_1^1
		&\leq \frac{M}{N}\int_{0}^{T} \|\nabla
		\left(\Delta u^i-\mathbb{P}(b(Nt)u\cdot \partial_i u\right)\|_{L^2(\mathbb{R}^3)} \|\nabla (\partial_i u^j_\epsilon)\|_{L^2(\mathbb{R}^3)} \|\Delta u_\epsilon^j\|_{L^\infty(\mathbb{R}^3)} dt
		\\
		&\lesssim \frac{M}{N}\int_{0}^{T} \| u\|_{\dot{H}^3(\mathbb{R}^3)} \| u_\epsilon\|_{\dot{H}^2(\mathbb{R}^3)} \|\Delta u_\epsilon\|_{L^\infty(\mathbb{R}^3)} 
		+  \|\nabla \mathbb{P} (u\cdot \nabla u)\|_{L^2(\mathbb{R}^3)} \| u_\epsilon\|_{\dot{H}^2(\mathbb{R}^3)} \|\Delta u_\epsilon\|_{L^\infty(\mathbb{R}^3)}  dt.
	\end{align*}
	Applying Lemma \ref{epsilon} and the property of Leray projection, we have
	\begin{align*}
		\mathcal{J}_1^1\lesssim \frac{M}{N\epsilon^2}\int_{0}^{T} \| u\|_{\dot{H}^3(\mathbb{R}^3)} \| u\|_{\dot{H}^1(\mathbb{R}^3)} \|\nabla u\|_{L^\infty(\mathbb{R}^3)} dt
		+  \frac{M}{N\epsilon^2} \int_{0}^{T} \|u \cdot \nabla u\|_{\dot{H}^1(\mathbb{R}^3)} \|u\|_{\dot{H}^1(\mathbb{R}^3)} \|\nabla u\|_{L^\infty(\mathbb{R}^3)} dt.
	\end{align*}
	For the rest term $\mathcal{J}^1_2$ and $\mathcal{J}^1_3$, we simply note that
	\begin{align*}
		\mathcal{J}^1_2+\mathcal{J}^1_3
		&\lesssim \frac{M}{N}\int_{0}^{T} \|\nabla^2
		\left(\Delta u_\epsilon-\mathbb{P}(b(Nt)u_\epsilon\cdot \nabla u_\epsilon)\right)\|_{L^2(\mathbb{R}^3)} \|\nabla u\|_{L^2(\mathbb{R}^3)} \|\nabla^2 u_\epsilon\|_{L^\infty(\mathbb{R}^3)} dt\\
		&\lesssim \frac{M}{N} \int_{0}^{T} \| u_\epsilon\|_{\dot{H}^4(\mathbb{R}^3)} \| u\|_{\dot{H}^1(\mathbb{R}^3)} \|\nabla^2 u^\epsilon\|_{L^\infty(\mathbb{R}^3)} 
		+  \|\nabla^2 \mathbb{P} (u_\epsilon\cdot \nabla u_\epsilon)\|_{L^2(\mathbb{R}^3)} \| u\|_{\dot{H}^1(\mathbb{R}^3)} \|\nabla^2 u^\epsilon\|_{L^\infty(\mathbb{R}^3)}   dt\\
		&	\lesssim\frac{M}{N\epsilon^2}\int_{0}^{T} \| u\|_{\dot{H}^3(\mathbb{R}^3)} \| u\|_{\dot{H}^1(\mathbb{R}^3)} \|\nabla u\|_{L^\infty(\mathbb{R}^3)} dt
		+  \frac{M}{N\epsilon^2} \int_{0}^{T} \|u \cdot \nabla u\|_{\dot{H}^1(\mathbb{R}^3)} \|u\|_{\dot{H}^1(\mathbb{R}^3)} \|\nabla u\|_{L^\infty(\mathbb{R}^3)} dt,
	\end{align*}
	where in the last line we use Lemma $\ref{epsilon}$.
	Therefore, combining \eqref{9/4} and \eqref{5/2} in Lemma \ref{interpolation}, we conclude that 
	\begin{align*}
		\mathcal{J}^1&\lesssim\frac{M}{N\epsilon^2}\|u\|_{L^\infty_T L^2(\mathbb{R}^3)}^{\frac{3}{4}} 
		\left(
		\| u\|_{L^\infty_T\dot{H}^2(\mathbb{R}^3)}^{\frac{9}{4}} + \| u\|_{L_T^2\dot{H}^3(\mathbb{R}^3)}^{\frac{9}{4}}\right)+\frac{M}{N\epsilon^2}\|u\|_{L_T^\infty L^2(\mathbb{R}^3)}^\frac{3}{2} 
		\left(
		\|u\|_{L_T^\infty \dot{H}^2(\mathbb{R}^3)}^{\frac{5}{2}}
		+ \|u\|_{L_T^2 \dot{H}^3(\mathbb{R}^3)}^{\frac{5}{2}}
		\right)\\
		&\lesssim\frac{M}{N\epsilon^2}\left(\|u_0\|_{L^2(\mathbb{R}^3)}^{\frac{3}{4}} 	X_T^{\frac{9}{8}}
		+\|u_0\|_{ L^2(\mathbb{R}^3)}^\frac{3}{2} 
		X_T^{\frac{5}{4}}\right),
	\end{align*}
	here in the last line we use  energy inequality,
	\begin{equation}
		\|u\|_{L_T^\infty L^2(\mathbb{R}^3)}^2+\int_{0}^{\infty}\|u(t)\|_{\dot{H}^1(\mathbb{R}^3)}^2 dt\leq 2\|u_0\|_{L^2(\mathbb{R}^3)}^2.
	\end{equation}
	As for the term $\mathcal{J}^2$,
	\begin{align*}
		\mathcal{J}^2=&	-\frac{1}{N} 	\int_{\mathbb{R}^3} \left(\int_{0}^{t} b(N\tau) d\tau\right) \nabla (\partial_i u^j_\epsilon)\cdot \nabla u^i\Delta u^j_\epsilon dx\vert_{t=0}^{t=T}.
	\end{align*}
	According to Gagliardo-Nirenberg interpolation inequality,
	\begin{align*}
		\|u\|_{L^\infty(\mathbb{R}^3)} \lesssim \|u\|_{L^2(\mathbb{R}^3)}^{\frac{1}{4}} \| u\|_{\dot{H}^2(\mathbb{R}^3)}^{\frac{3}{4}},
	\end{align*}
	combining with H\"{o}der inequality and Lemma \ref{epsilon}, it is obvious that
	\begin{align*}
		\int_{\mathbb{R}^3} |\nabla^2 u_{0,\epsilon}||\nabla u_0||\Delta u_{0,\epsilon}| dx
		&\leq \|\nabla^2 u_{0,\epsilon}\|_{L^\infty(\mathbb{R}^3)} \|\nabla^2 u_{0,\epsilon}\|_{L^2(\mathbb{R}^3)} \|\nabla u_0\|_{L^2(\mathbb{R}^3)} \\
		&\leq \epsilon^{-2} \|u_0\|_{L^\infty(\mathbb{R}^3)} \|u_0\|_{\dot{H}^2(\mathbb{R}^3)} \|u_0\|_{\dot{H}^1(\mathbb{R}^3)}\\
		&\lesssim \epsilon^{-2} \|u_0\|_{L^2(\mathbb{R}^3)}^{\frac{1}{4}} \|u_0\|_{\dot{H}^2(\mathbb{R}^3)}^{\frac{7}{4}} \|u_0\|_{\dot{H}^1(\mathbb{R}^3)}\\
		&\lesssim\epsilon^{-2} \|u_0\|_{H^2(\mathbb{R}^3)}^3,
	\end{align*}    
	and applying \eqref{nabla u2}, we infer
	\begin{align*}
		\int_{\mathbb{R}^3} |\nabla^2 u_\epsilon(T)||\Delta u_\epsilon(T)| |\nabla u(T)| dx
		&\leq \|\nabla^2 u_\epsilon(T)\|_{L^2(\mathbb{R}^3)} \|\Delta u_\epsilon(T)\|_{L^\infty(\mathbb{R}^3)} \|\nabla u(T)\|_{L^2(\mathbb{R}^3)}\\
		&\leq \epsilon^{-2} \|u(T)\|_{\dot{H}^2(\mathbb{R}^3)} \|u(T)\|_{L^\infty(\mathbb{R}^3)} \|u(T)\|_{\dot{H}^1(\mathbb{R}^3)}\\
		&\lesssim \epsilon^{-2} \|u(T)\|_{L^2(\mathbb{R}^3)}^{\frac{3}{4}} \|u(T)\|_{\dot{H}^2(\mathbb{R}^3)}^{\frac{9}{4}}\\
		&\lesssim \epsilon^{-2} \|u_0\|_{L^2(\mathbb{R}^3)}^{\frac{3}{4}} X_T^{\frac{9}{8}}.
	\end{align*}
	Thus, we get
	\begin{align*}
		\mathcal{J}^2 \lesssim \frac{M}{N\epsilon^2} \left(\|u_0\|_{H^2(\mathbb{R}^3)}^3+\|u_0\|_{L^2(\mathbb{R}^3)}^{\frac{3}{4}} X_T^{\frac{9}{8}}\right).
	\end{align*}
	Therefore, we arrive at
	\begin{align*}
		\|\Delta u(T)\|_{L^2(\mathbb{R}^3)}^2	+\int_{0}^{T}\|\Delta u(t)\|_{\dot{H}^1(\mathbb{R}^3)}^2
		\lesssim \|u_0\|_{\dot{H}^2(\mathbb{R}^3)}^2+ \epsilon M \|u_0\|_{L^2(\mathbb{R}^3)}^{\frac{1}{4}} X_T^{\frac{11}{8}}+\frac{M}{N\epsilon^2}\left(\|u_0\|_{ L^2(\mathbb{R}^3)}^{\frac{3}{4}} 
		X_T^{\frac{9}{8}}+\|u_0\|_{ L^2(\mathbb{R}^3)}^\frac{3}{2} 
		X_T^{\frac{5}{4}}+\|u_0\|_{H^2(\mathbb{R}^3)}^3\right).
	\end{align*}	
	As a result, it suffices to gain the expression of $X_T$:
	\begin{align*}
		X_T
		&\leq \underset{0\leq t \leq T}{\sup} \|\Delta u(T)\|_{L^2(\mathbb{R}^3)}^2	+\int_{0}^{T}\|\Delta u(t)\|_{\dot{H}^1(\mathbb{R}^3)}^2\\
		&\lesssim \|u_0\|_{\dot{H}^2(\mathbb{R}^3)}^2+ \epsilon M \|u_0\|_{L^2(\mathbb{R}^3)}^{\frac{1}{4}} X_T^{\frac{11}{8}}+\frac{M}{N\epsilon^2}\left(\|u_0\|_{ L^2(\mathbb{R}^3)}^{\frac{3}{4}} 
		X_T^{\frac{9}{8}}+\|u_0\|_{ L^2(\mathbb{R}^3)}^\frac{3}{2} 
		X_T^{\frac{5}{4}}+\|u_0\|_{H^2(\mathbb{R}^3)}^3\right).
	\end{align*}
	Taking $\epsilon=\left(\|u_0\|_{ L^2(\mathbb{R}^3)}^{\frac{3}{4}} 
	X_T^{\frac{9}{8}}+\|u_0\|_{ L^2(\mathbb{R}^3)}^\frac{3}{2} 
	X_T^{\frac{5}{4}}+\|u_0\|_{H^2(\mathbb{R}^3)}^3\right)^{\frac{1}{3}}\left(N\|u_0\|_{L^2(\mathbb{R}^3)}^{\frac{1}{4}} X_T^{\frac{11}{8}}\right)^{-\frac{1}{3}}$, we gain that there exists a constant $C>1$ such that
	\begin{align*}
		X_T\leq C \left(\|u_0\|_{\dot{H}^2(\mathbb{R}^3)}^2+\frac{M}{\sqrt[3]{N}}\left(\|u_0\|_{ L^2(\mathbb{R}^3)}^{\frac{3}{4}} 
		X_T^{\frac{9}{8}}+\|u_0\|_{ L^2(\mathbb{R}^3)}^\frac{3}{2} 
		X_T^{\frac{5}{4}}+\|u_0\|_{H^2(\mathbb{R}^3)}^3\right)^{\frac{1}{3}}\|u_0\|_{L^2(\mathbb{R}^3)}^{\frac{1}{6}} X_T^{\frac{11}{12}}\right).
	\end{align*}
\end{proof}
Now, we are at the position to give a proof of Theorem \ref{thm1.1}.
\begin{proof}[Proof of Theorem \ref{thm1.1}]
	Denote a set
	\begin{align}\label{set}
		\mathcal{T}:=\sup\{T:X_t\leq2C\|u_0\|_{H^2(\mathbb{R}^3)}^2, \quad\forall t\in[0,T]\}.
	\end{align}
	It is obvious that $\mathcal{T}>0$.  Now we claim that $\mathcal{T}=\infty$, otherwise
	taking $N_0=500M^3C^8\left(\|u_0\|_{H^2(\mathbb{R}^3)}+1\right)^4
	$ such that $N\geq N_0$, we have
	\begin{align*}
		X_T\leq \frac{3C}{2}  \|u_0\|_{H^2(\mathbb{R}^3)}^2,
	\end{align*}
	which yields a contradiction. Thus, we can derive that the set $\mathcal{T}=+\infty$.\\
	Consequently, for any $T>0$,
	\begin{align*}
		\|u\|_{L_T^\infty \dot{H}^2(\mathbb{R}^3)}^2 + \|u\|_{L_T^\infty \dot{H}^3(\mathbb{R}^3)}^2
		\leq 2C\|u_0\|_{H^2(\mathbb{R}^3)}^2,
	\end{align*}
	which means the equation \eqref{NS} admits global smooth solution.
\end{proof}

\section{Global large solution to  2D supercritical Surface Quasi-Geostrophic equations with time oscillation}
\begin{proposition}\label{proposition2}
	Let $\theta$ be a solution of \eqref{SQG} and denote
	\begin{equation}
		\widetilde{X}_T=\|\theta\|_{L^\infty_T\dot{H}^2(\mathbb{R}^2)}^2+\|\theta\|_{L^2_T\dot{H}^{2+\frac{\alpha}{2}}(\mathbb{R}^3)}^2,
	\end{equation}
	then we have 
	\begin{align}
		\widetilde{X}_T \lesssim
		&\left(
		\|\theta_0\|_{\dot{H}^2(\mathbb{R}^2)}^2
		+ \frac{1}{N^{\frac{\alpha}{4+\alpha}}}
		\left(
		\|\theta_0\|_{L^2(\mathbb{R}^2)} \widetilde{X}_T
		+ \|\theta_0\|_{L^2 (\mathbb{R}^2)}^{1+\frac{\alpha}{2}}
		\widetilde{X}_T^{\frac{6-\alpha}{4}}
		+ \|\theta_0\|_{L^2 (\mathbb{R}^2)}^{1+\frac{\alpha}{4}}
		\widetilde{X}_T^{\frac{8-\alpha}{8}}
		+ \|\theta_0\|_{H^2(\mathbb{R}^2)}
		\right)^{\frac{\alpha}{4+\alpha}}\cdot\right.\nonumber\\
		&\left.\quad \cdot\|\theta_0\|_{L^2(\mathbb{R})}^{\frac{\alpha}{4+\alpha}}
		\widetilde{X}_T^{\frac{12-\alpha}{8+2\alpha}}	
		\right).
	\end{align}
\end{proposition}
\begin{proof}
	Multiply $\Delta^2 \theta$ as test function on \eqref{SQG} and integral over $dx$, we get 
	\begin{align}\label{dx2}
		\frac{1}{2}\partial_t \|\Delta\theta(t)\|_{L^2(\mathbb{R}^2)}^2+\|\Delta^{1+\frac{\alpha}{4}}\theta(t)\|_{L^2(\mathbb{R}^2)}^2=-b(Nt)\int_{\mathbb{R}^2} \nabla^{\perp} (-\Delta)^{-\frac{1}{2}} \theta\cdot\nabla \theta \Delta^2 \theta dx.
	\end{align}
	For the term on right hand-side, integral by parts, we write
	\begin{align*}
		&-b(Nt)\int_{\mathbb{R}^2} \nabla^{\perp} (-\Delta)^{-\frac{1}{2}} \theta\cdot\nabla \theta \Delta^2 \theta dx\\
		&\qquad=-b(Nt) \int_{\mathbb{R}^2} \Delta
		\left(
		\nabla^{\perp} (-\Delta)^{-\frac{1}{2}} \theta\cdot\nabla \theta 
		\right) \Delta \theta dx\\
		&\qquad=-b(Nt) \int_{\mathbb{R}^2} \nabla^\perp (-\Delta)^{-\frac{1}{2}} \theta \cdot \nabla (\Delta \theta) \Delta \theta dx
		+b(Nt) \int_{\mathbb{R}^2} \nabla^\perp (-\Delta)^{-\frac{1}{2}} \theta \cdot \nabla \theta \Delta \theta dx\\
		&\qquad\quad+b(Nt) \int_{\mathbb{R}^2} (-\Delta)^{-\frac{1}{2}} \nabla (\partial_2 \theta) \cdot \nabla(\partial_1 \theta) \Delta \theta dx
		- b(Nt) \int_{\mathbb{R}^2} (-\Delta)^{-\frac{1}{2}} \nabla (\partial_1 \theta) \cdot \nabla(\partial_2 \theta) \Delta \theta dx.
	\end{align*} 
	For the first term, considering integral by parts, we compute as follows:
	\begin{align*}
		&-b(Nt) \int_{\mathbb{R}^2} \nabla^\perp (-\Delta)^{-\frac{1}{2}} \theta \cdot \nabla (\Delta \theta) \Delta \theta dx\\
		&\quad=b(Nt) \int_{\mathbb{R}^2} \nabla \cdot 
		\left(
		\left(\nabla^\perp (-\Delta)^{-\frac{1}{2}} \theta \right) \Delta \theta 
		\right) \Delta \theta dx\\
		&\quad=b(Nt) \int_{\mathbb{R}^2}  \Delta \theta \nabla \cdot 
		\left(
		\nabla^\perp (-\Delta)^{-\frac{1}{2}} \theta
		\right) \Delta \theta dx
		+b(Nt) \int_{\mathbb{R}^2} \nabla^\perp (-\Delta)^{-\frac{1}{2}} \theta \cdot \nabla(\Delta \theta) \Delta \theta.
	\end{align*}	
	Since $\nabla \cdot \nabla^\perp=0$, we obtain that the first term equals $0$, i.e.
	\begin{align*}
		-b(Nt) \int_{\mathbb{R}^2} \nabla^\perp (-\Delta)^{-\frac{1}{2}} \theta \cdot \nabla (\Delta \theta) \Delta \theta dx
		=b(Nt) \int_{\mathbb{R}^2} \nabla^\perp (-\Delta)^{-\frac{1}{2}} \theta \cdot \nabla (\Delta \theta) \Delta \theta dx=0.
	\end{align*}
	Now we take integral over time about the equality \eqref{dx2}, we have
	\begin{align*}
		&\frac{1}{2}\|\Delta\theta(T)\|_{L^2(\mathbb{R}^2)}^2	-\frac{1}{2}\|\Delta\theta_0\|_{L^2(\mathbb{R}^2)}^2+\int_{0}^{T}\|\Delta^{1+\frac{\alpha}{4}}\theta(t)\|_{L^2(\mathbb{R}^2)}^2\\
		&\quad=\int_{0}^{T} \int_{\mathbb{R}^2} b(Nt) \nabla^\perp (-\Delta)^{\frac{1}{2}} \theta \cdot \nabla \theta \Delta \theta dx dt
		+ \int_{0}^{T} \int_{\mathbb{R}^2} b(Nt) (-\Delta)^{-\frac{1}{2}} \nabla(\partial_2 \theta) \cdot \nabla(\partial_1 \theta) \Delta \theta dx dt\\
		&\quad\quad- \int_{0}^{T} \int_{\mathbb{R}^2} b(Nt) (-\Delta)^{-\frac{1}{2}} \nabla(\partial_1 \theta) \cdot \nabla(\partial_2 \theta) \Delta \theta dx dt,
	\end{align*}	
	Let $\theta_\epsilon(t,x)=\left(\theta(t)\star \rho_\epsilon\right)(x)$, where $\rho_\epsilon(x)=\epsilon^{-2}\rho(x/\epsilon)$ is a modifier. Then we divide the right hand-side into two parts,
	\begin{align*}
		RHS=\mathbb{I}+\mathbb{J},
	\end{align*}
	where
	\begin{align*}
		\mathbb{I}=&- \int_{0}^{T} \int_{\mathbb{R}^2} b(Nt)
		\left(
		\nabla^\perp (-\Delta)^{\frac{1}{2}} \theta \cdot \nabla \theta \Delta \theta 
		- \nabla^\perp (-\Delta)^{\frac{1}{2}} \theta_\epsilon \cdot \nabla \theta \Delta \theta _\epsilon
		\right) dx dt\\
		&\quad+\int_{0}^{T} \int_{\mathbb{R}^2} b(Nt)
		\left(
		(-\Delta)^{-\frac{1}{2}} \nabla(\partial_2 \theta) \cdot \nabla(\partial_1 \theta) \Delta \theta
		- (-\Delta)^{-\frac{1}{2}} \nabla(\partial_2 \theta) \cdot \nabla(\partial_1 \theta_\epsilon) \Delta \theta_\epsilon
		\right) dx dt\\
		&\quad-\int_{0}^{T} \int_{\mathbb{R}^2} b(Nt)
		\left(
		(-\Delta)^{-\frac{1}{2}} \nabla(\partial_1 \theta) \cdot \nabla(\partial_2 \theta) \Delta \theta
		- (-\Delta)^{-\frac{1}{2}} \nabla(\partial_1 \theta) \cdot \nabla(\partial_2 \theta_\epsilon) \Delta \theta_\epsilon
		\right)dx dt,\\
		\mathbb{J}=&-\int_{0}^{T} \int_{\mathbb{R}^2} b(Nt) \nabla^\perp (-\Delta)^{\frac{1}{2}} \theta_\epsilon \cdot \nabla \theta \Delta \theta_\epsilon dx dt
		+\int_{0}^{T} \int_{\mathbb{R}^2} b(Nt) (-\Delta)^{-\frac{1}{2}} \nabla(\partial_2 \theta) \cdot \nabla(\partial_1 \theta_\epsilon) \Delta \theta_\epsilon dx dt\\
		&\quad-\int_{0}^{T} \int_{\mathbb{R}^2} b(Nt) (-\Delta)^{-\frac{1}{2}} \nabla(\partial_1 \theta) \cdot \nabla(\partial_2 \theta_\epsilon) \Delta \theta_\epsilon dx dt.
	\end{align*}
	Without loss of generality, we only show the estimate of the first term of $\mathbb{I}$ and $\mathbb{J}$, the estimates of the other terms are similar, we omit the details.\\
	For the first term of $\mathbb{I}$,  applying H\"{o}lder inequality, Lemma \ref{epsilon} and \ref{modifier} we deduce that
	\begin{align*}
		&-\int_{0}^{T} \int_{\mathbb{R}^2} b(Nt)
		\left(
		\nabla^\perp (-\Delta)^{\frac{1}{2}} \theta \cdot \nabla \theta \Delta \theta 
		- \nabla^\perp (-\Delta)^{\frac{1}{2}} \theta_\epsilon \cdot \nabla \theta \Delta \theta _\epsilon
		\right) dx dt\\
		&\quad=-\int_{0}^{T} \int_{\mathbb{R}^2} b(Nt)
		\left(
		\nabla^\perp (-\Delta)^{\frac{1}{2}} \theta \cdot \nabla \theta \Delta \theta 
		- \nabla^\perp (-\Delta)^{\frac{1}{2}} \theta_\epsilon \cdot \nabla \theta \Delta \theta
		\right) dx dt\\
		&\qquad\quad- \int_{0}^{T} \int_{\mathbb{R}^2} b(Nt)
		\left(
		\nabla^\perp (-\Delta)^{\frac{1}{2}} \theta_\epsilon \cdot \nabla \theta \Delta \theta 
		- \nabla^\perp (-\Delta)^{\frac{1}{2}} \theta_\epsilon \cdot \nabla \theta \Delta \theta _\epsilon
		\right) dx dt\\
		&\quad\leq M \left(\int_{0}^{T}  
		\|\nabla^\perp (-\Delta)^{\frac{1}{2}} \left(\theta -  \theta_\epsilon\right) \|_{L^2(\mathbb{R}^2)}
		\|\nabla \theta\|_{L^\infty(\mathbb{R}^2)} \|\Delta \theta\|_{L^2(\mathbb{R}^2)}dt\right.\\
		&\left.\qquad\qquad+ \int_{0}^{T} 
		\|\nabla^\perp (-\Delta)^{\frac{1}{2}} \theta_\epsilon \|_{L^2(\mathbb{R}^2)} 
		\|\nabla \theta\|_{L^\infty(\mathbb{R}^2)} 
		\|\Delta \theta-\Delta \theta_\epsilon\|_{L^2(\mathbb{R}^2)}dt\right)\\
		&\quad\lesssim\epsilon^{\frac{\alpha}{2}} M \int_{0}^{T} \|\theta\|_{\dot{H}^{2+\frac{\alpha}{2}}(\mathbb{R}^2)}
		\|\nabla \theta\|_{L^\infty(\mathbb{R}^2)} \|\theta\|_{\dot{H}^2(\mathbb{R}^2)} dt\\
		&\quad\lesssim  \epsilon^{\frac{\alpha}{2}} M \|\theta_0\|_{L^2(\mathbb{R}^2)} ^{\frac{\alpha}{4}}
		{\widetilde{X}_T}^{\frac{12-\alpha}{8}},
	\end{align*}
	where in the last line we use inequality \eqref{estimate 1} and energy inequality
	\begin{align*}
		\|\theta(t,\cdot)\|_{L_T L^2(\mathbb{R}^2)}^2
		+2\int_{0}^T \|\theta(s,\cdot)\|_{\dot{H}^{\frac{\alpha}{2}}(\mathbb{R}^2)}^2 ds
		\leq \|\theta_0\|_{L^2(\mathbb{R}^2)}^2.
	\end{align*}	
	Using similar estimate, we obtain that
	\begin{align*}
		\mathbb{I}\lesssim  \epsilon^{\frac{\alpha}{2}} M \|\theta_0\|_{L^2(\mathbb{R}^2)} ^{\frac{\alpha}{4}}
		{\widetilde{X}_T}^{\frac{12-\alpha}{8}}.
	\end{align*}	
	For the first term of $\mathbb{J}$, we integral by parts over $dt$ to get
	\begin{align*}
		-\int_{0}^{T} \int_{\mathbb{R}^2} b(Nt) \nabla^\perp (-\Delta)^{\frac{1}{2}} \theta_\epsilon \cdot \nabla \theta \Delta \theta_\epsilon dx dt
		&=\frac{1}{N} 	\int_{0}^{T} \int_{\mathbb{R}^2} \left(\int_{0}^{t} b(N\tau)d\tau\right) 
		\partial_t [\nabla^\perp (-\Delta)^{\frac{1}{2}} \theta_\epsilon \cdot \nabla \theta \Delta \theta_\epsilon] dx dt\\
		&\qquad- \frac{1}{N}  \int_{\mathbb{R}^2} \left(\int_{0}^{t} b(N\tau)d\tau\right) \nabla^\perp (-\Delta)^{\frac{1}{2}} \theta_\epsilon \cdot \nabla \theta \Delta \theta_\epsilon dx\big\vert_{t=0}^{t=T}\\
		&=:\mathbb{J}^1+\mathbb{J}^2.
	\end{align*}
	As for the term $\mathbb{J}^1$, we write
	\begin{align*}
		\mathbb{J}^1&=\frac{1}{N} 	\int_{0}^{T} \int_{\mathbb{R}^2} \left(\int_{0}^{t} b(N\tau)d\tau\right) 
		\nabla^\perp (-\Delta)^{\frac{1}{2}} \left( \partial_t \theta_\epsilon\right) \cdot \nabla \theta \Delta \theta_\epsilon dx dt\\
		&\quad+\frac{1}{N} 	\int_{0}^{T} \int_{\mathbb{R}^2} \left(\int_{0}^{t} b(N\tau)d\tau\right) 
		\nabla^\perp (-\Delta)^{\frac{1}{2}}  \theta_\epsilon \cdot \nabla \left(\partial_t \theta\right) \Delta \theta_\epsilon dx dt\\
		&\quad+\frac{1}{N} 	\int_{0}^{T} \int_{\mathbb{R}^2} \left(\int_{0}^{t} b(N\tau)d\tau\right) 
		\nabla^\perp (-\Delta)^{\frac{1}{2}}  \theta_\epsilon \cdot \nabla  \theta \Delta \left(\partial_t \theta_\epsilon\right) dx dt\\
		&=:\mathbb{J}_1^1+\mathbb{J}^1_2+\mathbb{J}^1_3.
	\end{align*}
	To treat the term $\mathbb{J}_1^1$, we make use of the fact that
	$\partial_t \theta=-(-\Delta)^{\frac{\alpha}{2}} \theta-b(Nt)\nabla^\perp (-\Delta)^{-\frac{1}{2}} \theta \cdot \nabla\theta$,
	thus we have
	\begin{align*}
		\mathbb{J}_1^1&=-\frac{1}{N} 	\int_{0}^{T} \int_{\mathbb{R}^2} \left(\int_{0}^{t} b(N\tau)d\tau\right) 
		\left(\nabla^\perp (-\Delta)^{\frac{1}{2}} (-\Delta)^{\frac{\alpha}{2}} \theta_\epsilon\right)
		\cdot \nabla \theta \Delta \theta_\epsilon dx dt\\
		&\quad\quad- \frac{1}{N} \int_{0}^{T} \int_{\mathbb{R}^2} \left(\int_{0}^{t} b(N\tau)d\tau\right) b(Nt) \nabla^\perp (-\Delta)^{\frac{1}{2}}
		\left(
		\nabla^\perp (-\Delta)^{-\frac{1}{2}} \theta_\epsilon \cdot \nabla \theta_\epsilon
		\right)
		\cdot \nabla \theta \Delta \theta_\epsilon dx dt\\
		&=:\mathbb{J}_{1,1}^1+\mathbb{J}_{1,2}^1.
	\end{align*}	
	According to H\"{o}lder inequality, Lemma \ref{epsilon} and \eqref{estimate 2}, we yield
	\begin{align*}
		\mathbb{J}_{1,1}^1 &\leq \frac{M}{N} \int_{0}^T  
		\|\nabla^\perp (-\Delta)^{\frac{1+\alpha}{2}} \theta_\epsilon\|_{L^2(\mathbb{R}^2)}
		\|\nabla \theta\|_{L^\infty (\mathbb{R}^2)}
		\|\Delta \theta_\epsilon\|_{L^2(\mathbb{R}^2)}dt\\
		&\leq \frac{M}{N\epsilon^2} \int_{0}^T 
		\|\theta\|_{\dot{H}^{1+\alpha}(\mathbb{R}^2)} 
		\|\nabla \theta\|_{L^\infty (\mathbb{R}^2)}
		\|\theta\|_{\dot{H}^1(\mathbb{R}^2)} dt\\
		&\lesssim \frac{M}{N\epsilon^2} \|\theta_0\|_{L^2(\mathbb{R}^2)} \widetilde{X}_T,
	\end{align*}	
	with the application of \eqref{estimate 3}, we have
	\begin{align*}
		\mathbb{J}_{1,2}^1&\leq \frac{M^2}{N} \int_0^T 
		\|\nabla^\perp (-\Delta)^{-\frac{1}{2}}
		\left(
		\nabla^\perp (-\Delta)^{-\frac{1}{2}} \theta_\epsilon \cdot \nabla \theta_\epsilon
		\right)\|_{L^2(\mathbb{R}^2)}
		\|\nabla \theta\|_{L^\infty(\mathbb{R}^2)}
		\|\Delta \theta_\epsilon\|_{L^2(\mathbb{R}^2)} dx dt\\
		&\leq \frac{M^2}{N} \int_{0}^T 
		\left(
		\|\nabla^2 \theta_\epsilon\|_{L^\infty(\mathbb{R}^2)}
		\|\nabla \theta_\epsilon\|_{L^2(\mathbb{R}^2)}
		+ \|\nabla^2 \theta_\epsilon\|_{L^2(\mathbb{R}^2)}
		\|\nabla \theta_\epsilon\|_{L^\infty(\mathbb{R}^2)}
		\right)
		\|\nabla \theta\|_{L^\infty(\mathbb{R}^2)}
		\|\Delta \theta_\epsilon\|_{L^2(\mathbb{R}^2)} dt\\
		&\leq \frac{M^2}{N\epsilon^2} \int_{0}^T
		\|\nabla \theta\|_{L^\infty(\mathbb{R}^2)}^2
		\|\nabla \theta\|_{L^2(\mathbb{R}^2)}^2 dt\\
		&\lesssim \frac{M^2}{N\epsilon^2} \|\theta_0\|_{L^2(\mathbb{R}^2)}^{1+\frac{\alpha}{2}}
		{\widetilde{X}_T}^{\frac{6-\alpha}{4}}.
	\end{align*}	
	Without loss of generality, applying the same method, it is easy to check 
	\begin{align*}
		&\mathbb{J}^1_2 \lesssim 
		\frac{M}{N\epsilon^2} \|\theta_0\|_{L^2(\mathbb{R}^2)}
		\widetilde{X}_T
		+\frac{M^2}{N\epsilon^2} \|\theta_0\|_{L^2(\mathbb{R}^2)}^{1+\frac{\alpha}{2}}
		{\widetilde{X}_T}^{\frac{6-\alpha}{4}},\\
		&\mathbb{J}^1_3 \lesssim 
		\frac{M}{N\epsilon^2} \|\theta_0\|_{L^2(\mathbb{R}^2)}
		\widetilde{X}_T
		+\frac{M^2}{N\epsilon^2} \|\theta_0\|_{L^2(\mathbb{R}^2)}^{1+\frac{\alpha}{2}}
		{\widetilde{X}_T}^{\frac{6-\alpha}{4}},
	\end{align*}	
	we omit the details here.\\
	Hence, we have
	\begin{align*}
		\mathbb{J}^1 \lesssim 
		\frac{M}{N\epsilon^2} \|\theta_0\|_{L^2(\mathbb{R}^2)}
		\widetilde{X}_T
		+\frac{M^2}{N\epsilon^2} \|\theta_0\|_{L^2(\mathbb{R}^2)}^{1+\frac{\alpha}{2}}
		{\widetilde{X}_T}^{\frac{6-\alpha}{4}}.
	\end{align*}	
	As for the term $\mathbb{J}^2$,
	\begin{align*}
		\mathbb{J}^2=
		- \frac{1}{N}  \int_{\mathbb{R}^2} \left(\int_{0}^{t} b(N\tau)d\tau\right) \nabla^\perp (-\Delta)^{\frac{1}{2}} \theta_\epsilon \cdot \nabla \theta \Delta \theta_\epsilon dx\big\vert_{t=0}^{t=T}.
	\end{align*}	
	According to Gagliardo-Nirenberg interpolation,
	\begin{align*}
		\|\theta_0\|_{L^\infty (\mathbb{R}^2) }
		\lesssim \|\theta_0\|_{\dot{H}^2(\mathbb{R}^2)}^{\frac{1}{2}}
		\|\theta_0\|_{L^2(\mathbb{R}^2)}^{\frac{1}{2}},
	\end{align*}	
	combining with H\"{o}lder inequality and Lemma $\ref{epsilon}$, it is obvious that
	\begin{align*}
		\int_{\mathbb{R}^2} 
		|\nabla^\perp (-\Delta)^{\frac{1}{2}} \theta_{0,\epsilon}|
		|\nabla \theta_0|
		|\Delta \theta_{0,\epsilon}| dx
		&\leq	\|\nabla^\perp (-\Delta)^{\frac{1}{2}} \theta_{0,\epsilon}\|_{L^\infty (\mathbb{R}^2)}
		\|\nabla \theta_0\|_{L^2(\mathbb{R}^2)}
		\|\Delta \theta_{0,\epsilon}\|_{L^2(\mathbb{R}^2)}\\
		&\leq \epsilon^{-2}
		\|\theta_0\|_{L^\infty (\mathbb{R}^2)}
		\|\theta_0\|_{\dot{H}^1(\mathbb{R}^2)}
		\|\theta_0\|_{\dot{H}^2(\mathbb{R}^2)}\\
		&\leq \epsilon^{-2}
		\|\theta_0\|_{L^2 (\mathbb{R}^2)}^{\frac{1}{2}}
		\|\theta_0\|_{\dot{H}^1(\mathbb{R}^2)}
		\|\theta_0\|_{\dot{H}^2(\mathbb{R}^2)}^{\frac{3}{2}}\\
		&\lesssim \epsilon^{-2} \|\theta_0\|_{H^2(\mathbb{R}^2)}^3,
	\end{align*}	
	and applying \eqref{nabla theta}, we infer
	\begin{align*}
		\int_{\mathbb{R}^2} 
		|\nabla^\perp (-\Delta)^{\frac{1}{2}} \theta_\epsilon(T)|
		|\nabla \theta(T)|
		|\Delta \theta_\epsilon(T)| dx
		&\leq \|\nabla^\perp (-\Delta)^{\frac{1}{2}} \theta_\epsilon(T)\|_{L^2(\mathbb{R}^2)}
		\|\nabla \theta(T)\|_{L^\infty (\mathbb{R}^2)}
		\|\Delta \theta_\epsilon(T)\|_{L^2 (\mathbb{R}^2)}\\
		&\leq \epsilon^{-2} 
		\|\theta(T)\|_{\dot{H}^2(\mathbb{R}^2)}
		\|\nabla \theta(T)\|_{L^\infty(\mathbb{R}^2)}
		\|\theta(T)\|_{L^2(\mathbb{R}^2)}\\
		&\lesssim \epsilon^{-2}
		\|\theta(T)\|_{\dot{H}^2(\mathbb{R}^2)}
		\|\theta(T)\|_{\dot{H}^{2+\frac{\alpha}{2}}(\mathbb{R}^2)}^{1-\frac{\alpha}{4}}
		\|\theta(T)\|_{\dot{H}^{\frac{\alpha}{2}}(\mathbb{R}^2)}^{\frac{\alpha}{4}}
		\|\theta(T)\|_{L^2(\mathbb{R}^2)}\\
		&\lesssim \epsilon^{-2}
		\|\theta(T)\|_{\dot{H}^{\frac{\alpha}{2}}(\mathbb{R}^2)}^{\frac{\alpha}{4}}
		\|\theta(T)\|_{L^2(\mathbb{R}^2)}
		\left(
		\|\theta(T)\|_{\dot{H}^2(\mathbb{R}^2)}^{\frac{8-\alpha}{4}}
		+ \|\theta(T)\|_{\dot{H}^{2+\frac{\alpha}{2}}(\mathbb{R}^2)}^{\frac{8-\alpha}{4}}
		\right)\\
		&\lesssim \epsilon^{-2} \|\theta_0\|_{L^2(\mathbb{R}^2)}^{1+\frac{\alpha}{4}}
		{\widetilde{X}_T}^{\frac{8-\alpha}{8}}.
	\end{align*}	
	Thus, we get
	\begin{align*}
		\mathbb{J}^2 \lesssim \frac{M}{N\epsilon^2}
		\left(
		\|\theta_0\|_{H^2(\mathbb{R}^2)}^3
		+ \|\theta_0\|_{L^2 (\mathbb{R}^2)}^{1+\frac{\alpha}{4}}
		{\widetilde{X}_T}^{\frac{8-\alpha}{8}}
		\right).
	\end{align*}	
	Using similar estimate, we obtain that
	\begin{align*}
		\mathbb{J} \lesssim 
		\frac{M}{N\epsilon^2}
		\left( 
		\|\theta_0\|_{L^2(\mathbb{R}^2)}
		\widetilde{X}_T
		+\|\theta_0\|_{L^2(\mathbb{R}^2)}^{1+\frac{\alpha}{2}}
		\widetilde{X}_T^{\frac{6-\alpha}{4}}
		+ \|\theta_0\|_{L^2 (\mathbb{R}^2)}^{1+\frac{\alpha}{4}}
		\widetilde{X}_T^{\frac{8-\alpha}{8}}
		+ \|\theta_0\|_{H^2(\mathbb{R}^2)}^3
		\right).
	\end{align*}	
	Therefore, we arrive at
	\begin{align*}
		&\|\Delta\theta(T)\|_{L^2(\mathbb{R}^2)}^2	+\int_{0}^{T}\|\Delta^{1+\frac{\alpha}{4}}\theta(t)\|_{L^2(\mathbb{R}^2)}^2\\
		&\quad\lesssim  \|\theta_0\|_{\dot{H}^2(\mathbb{R}^2)}^2
		+\epsilon^{\frac{\alpha}{2}} M \|\theta_0\|_{L^2(\mathbb{R}^2)} ^{\frac{\alpha}{4}}
		\widetilde{X}_T^{\frac{12-\alpha}{8}}
		+\frac{M}{N\epsilon^2}
		\left( \|\theta_0\|_{L^2(\mathbb{R}^2)}
		\widetilde{X}_T
		+ \|\theta_0\|_{L^2(\mathbb{R}^2)}^{1+\frac{\alpha}{2}}
		\widetilde{X}_T^{\frac{6-\alpha}{4}}
		+ \|\theta_0\|_{L^2 (\mathbb{R}^2)}^{1+\frac{\alpha}{4}}
		\widetilde{X}_T^{\frac{8-\alpha}{8}}
		+ \|\theta_0\|_{H^2(\mathbb{R}^2)}^3
		\right).
	\end{align*}	
	As a result, it suffices to gain the expression of $\widetilde{X}_T$:
	\small{     \begin{align*}
			\widetilde{X}_T
			&\leq \underset{0\leq t \leq T}{\sup}
			\|\Delta\theta(T)\|_{L^2(\mathbb{R}^2)}^2	+\int_{0}^{T}\|\Delta^{1+\frac{\alpha}{4}}\theta(t)\|_{L^2(\mathbb{R}^2)}^2\\
			&\lesssim  \|\theta_0\|_{\dot{H}^2(\mathbb{R}^2)}^2
			+\epsilon^{\frac{\alpha}{2}} M \|\theta_0\|_{L^2(\mathbb{R}^2)} ^{\frac{\alpha}{4}}
			\widetilde{X}_T^{\frac{12-\alpha}{8}}
			+\frac{M}{N\epsilon^2}
			\left( \|\theta_0\|_{L^2(\mathbb{R}^2)}
			\widetilde{X}_T
			+ \|\theta_0\|_{L^2(\mathbb{R}^2)}^{1+\frac{\alpha}{2}}
			\widetilde{X}_T^{\frac{6-\alpha}{4}}
			+ \|\theta_0\|_{L^2 (\mathbb{R}^2)}^{1+\frac{\alpha}{4}}
			\widetilde{X}_T^{\frac{8-\alpha}{8}}
			+ \|\theta_0\|_{H^2(\mathbb{R}^2)}^3
			\right).
	\end{align*}	}
	Taking 
	$\epsilon=
	\left(
	\|\theta_0\|_{L^2 (\mathbb{R}^2)} \widetilde{X}_T
	+ \|\theta_0\|_{L^2 (\mathbb{R}^2)}^{1+\frac{\alpha}{2}}
	\widetilde{X}_T^{\frac{6-\alpha}{4}}
	+ \|\theta_0\|_{L^2 (\mathbb{R}^2)}^{1+\frac{\alpha}{4}}
	\widetilde{X}_T^{\frac{8-\alpha}{8}}
	+ \|\theta_0\|_{H^2(\mathbb{R}^2)}^3
	\right)^{\frac{2}{4+\alpha}}
	\left(
	N \|\theta_0\|_{L^2 (\mathbb{R}^2)}^{\frac{\alpha}{4}}
	\widetilde{X}_T^{\frac{12-\alpha}{8}}
	\right)^{-\frac{2}{4+\alpha}}$, we gain that there exists a constant $C>1$ such that 
	\small{     \begin{align*}
			\widetilde{X}_T \leq C
			&\left(
			\|\theta_0\|_{\dot{H}^2(\mathbb{R}^2)}^2
			+ \frac{M}{N^{\frac{\alpha}{4+\alpha}}}
			\left(
			\|\theta_0\|_{L^2(\mathbb{R}^2)} \widetilde{X}_T
			+ \|\theta_0\|_{L^2 (\mathbb{R}^2)}^{1+\frac{\alpha}{2}}
			\widetilde{X}_T^{\frac{6-\alpha}{4}}
			+ \|\theta_0\|_{L^2 (\mathbb{R}^2)}^{1+\frac{\alpha}{4}}
			\widetilde{X}_T^{\frac{8-\alpha}{8}}
			+ \|\theta_0\|_{H^2(\mathbb{R}^2)}^3
			\right)^{\frac{\alpha}{4+\alpha}}\cdot\right.\\
			&\left.\quad \cdot\|\theta_0\|_{L^2(\mathbb{R})}^{\frac{\alpha}{4+\alpha}}
			\widetilde{X}_T^{\frac{12-\alpha}{8+2\alpha}}	
			\right).
	\end{align*}	}
\end{proof}
Now, we are at the position to give a proof of Theorem \ref{thm 1.2}.
\begin{proof}[Proof of Theorem \ref{thm 1.2}]
	Applying the same proof method as Theorem $\ref{thm1.1}$, we can easily arrive at the following results:
	for any $T>0$,
	\begin{align*}
		\|\theta\|_{L^\infty_T\dot{H}^2(\mathbb{R}^2)}^2+\|\theta\|_{L^2_T\dot{H}^{2+\frac{\alpha}{2}}(\mathbb{R}^3)}^2
		\leq 2C\|\theta_0\|_{H^2(\mathbb{R}^2)}^2,
	\end{align*}	
	which means the equation \eqref{SQG} admits global smooth solution.
\end{proof}	\vspace{0.2cm}

\textbf{Acknowledgements:} The authors are grateful to Professor Quoc Hung Nguyen, who introduced this project to us and patiently guided, supported, and encouraged us during this work.

\end{document}